\documentclass[11pt,reqno]{amsart}

\usepackage{amssymb,amsmath,amscd, amsthm}
\usepackage{latexsym,bm,bbm,mathrsfs}
\usepackage{hyperref,graphicx,enumerate}
\usepackage{mathtools}
\usepackage{stmaryrd}
\usepackage[all,pdf]{xy}
\usepackage{graphicx}
\usepackage{tikz,circledsteps,tikz-cd}
\usepackage{color}
\usepackage[normalem]{ulem}
\graphicspath{ {./} }
\usepackage{stackrel}
\usepackage[left=30mm, right=30mm, top=30mm, bottom=30mm]{geometry}
\usepackage[shortlabels]{enumitem}
\usepackage{verbatim}

\allowdisplaybreaks
\newtheorem{thm}{Theorem}[section]

\newtheorem*{quest*}{Question}
\newtheorem{theorem}[thm]{Theorem}

\newtheorem{lemma}[thm]{Lemma}

\newtheorem{conj}[thm]{Conjecture}

\newtheorem*{theorem*}{Theorem}

\theoremstyle{definition}

\theoremstyle{remark}
\newtheorem{remark}[thm]{Remark}

\newcommand{\bbc}{{\mathbb{C}}}

\newcommand{\bbq}{{\mathbb{Q}}}

\newcommand{\bbz}{{\mathbb{Z}}}
\newcommand{\cO}{{\mathcal{O}}}
\newcommand{\Gal}{\operatorname{Gal}}

\newcommand{\Tr}{\operatorname{Tr}}

\newcommand{\tr}{\operatorname{Tr}}

\newtheoremstyle{named}{}{}{\itshape}{}{\bfseries}{.}{.5em}{\thmnote{#3}#1}
\theoremstyle{named}
\newtheorem*{namedtheorem}{}

\usepackage{enumitem}

\begin{document}

\title[Analytic Rank-One Elliptic Curves over Function Fields]{Analytic Rank-One Elliptic Curves over Function Fields and their rank over certain ring class fields}

\author{Seokhyun Choi}
\address{Department of Mathematical Sciences, KAIST, 291 Daehak-ro, Yuseong-gu, Daejeon, 34141, South Korea}
\email{sh021217@kaist.ac.kr}

\author{Bo-Hae Im}
\address{Department of Mathematical Sciences, KAIST, 291 Daehak-ro, Yuseong-gu, Daejeon, 34141, South Korea}
\email{bhim@kaist.ac.kr}
\thanks{Bo-Hae Im was supported by Basic Science Research Program through the National Research Foundation of Korea(NRF) grant funded by the Korea government(MSIT)(NRF-2023R1A2C1002385, or RS-2023-NR076333).}

\author{Beomho Kim}
\address{Department of Mathematical Sciences, KAIST, 291 Daehak-ro, Yuseong-gu, Daejeon, 34141, South Korea}
\email{bhkim@kaist.ac.kr}

\date{\today}

\maketitle

\begin{abstract}
    Let $E/k$ be a non-isotrivial elliptic curve over a global function field $k$ of characteristic $p>3$, and  $G\subset \Gal(k^{\mathrm{sep}}/k)$ be a topologically finitely generated subgroup. We prove that if $E/k$ has analytic rank $1$, then its rank over the fixed subfield $L^G$ is infinite, where $L$ is the infinite ring class extension of some finite separable extension $K/k$. If $E/k$ has analytic rank $0$, then we prove that the same holds provided there exists an imaginary quadratic extension $K/k$ such that $E/K$ has analytic rank $1$ and satisfies the Heegner hypothesis.
\end{abstract}

\section{Introduction}\label{sec_intro}
Elliptic curves are central objects in Diophantine geometry, which arise naturally from the study of Diophantine equations. A key feature is the group of rational points on an elliptic curve, which encodes deep information about its underlying arithmetic structure. According to the Mordell-Weil theorem (\cite[Theorem~VIII.6.7]{Sil09}, \cite[Theorem~III.6.1]{Sil94}), these points over a global field form a finitely generated abelian group, and thus have a finite rank.

The complexity increases when examining rank behavior across field extensions. While the rank stays finite over finite extensions of global fields, infinite extensions present a much more intricate picture where the rank can remain either finite or grow to infinity.  Understanding how and why this happens remains a significant open challenge in the study of elliptic curves.

Let $E/k$ be an elliptic curve over a global field $k$ and $K/k$ be an infinite field extension. Then, the rank of $E(K)$ is finite in cases such as when:
\begin{itemize}
    \item $k=\bbq$ and $K/k$ is a cyclotomic $\bbz_p$-extension (\cite{Maz72}, \cite{Kat04}),
\end{itemize}
while the rank of $E(K)$ is infinite in the following cases:
\begin{itemize}
    \item $k$ is a global field and $K$ is the separable closure of $k$ (\cite{FJ74}), 
    \item $k$ is a number field and $K$ is a maximal abelian extension of $k$ (\cite{FJ74}), 
    \item $k$ is an imaginary quadratic number field and $K/k$ is an anticyclotomic $\bbz_p$-extension~(\cite{CV07}), 
    \item $k$ is an imaginary quadratic extension of a global function field and $K/k$ is an anticyclotomic $\bbz_p$-extension (\cite{Bre04}).
\end{itemize}

This last result by Breuer \cite{Bre04} serves as the direct motivation for the present work. Breuer's construction in \cite{Bre04} is carried out in the ring class field tower, where he uses associated Heegner points to prove that they generate a subgroup of infinite rank. 
This raises a natural and highly non-trivial question: if we shrink this field -- specifically, by passing down to a smaller subfield fixed by a subgroup of the Galois group -- does this infinitude of rank persist, or does it collapse to a finite value?

A prominent conjecture concerning this phenomenon of rank survival over fixed subfields was formulated by M.~Larsen~\cite{Lar03}:
\begin{conj}\label{Larsen}
    Let $A/k$ be an abelian variety over a finitely generated infinite field $k$. Let $G$ be a topologically finitely generated subgroup of $G_k := \Gal(k^{\mathrm{sep}}/k)$. Then the rank of $A$ over the fixed subfield $(k^{\mathrm{sep}})^G$ of $k^{\mathrm{sep}}$ under $G$ is infinite.
\end{conj}

This conjecture has been approached through  various methods. 
First, employing techniques from Diophantine geometry, the second author and M.~Larsen \cite{IL08} proved Conjecture~\ref{Larsen} for procyclic groups $G$. 

The arithmetic approach via Heegner points has also yielded several partial results toward Conjecture~\ref{Larsen}. In \cite{Im07}, the second author used the Heegner point method vertically to prove the conjecture when $A$ is an elliptic curve, $k=\bbq$, and $G$ is procyclic. This method was generalized in \cite{BI08} to cases where $k$ is a totally real number field or a global function field. Recently, the first and second authors applied the Heegner point method horizontally to prove Conjecture~\ref{Larsen} when $A$ is an elliptic curve of analytic rank at most one over $k=\bbq$. 

Another approach was given by Tim and Vladimir Dokchitser~\cite{DD09}, which proves Conjecture~\ref{Larsen} when $A$ is an elliptic curve and $k=\bbq$, assuming either the rank part of the Birch and Swinnerton-Dyer conjecture or the finiteness of the Tate-Shafarevich group over arbitrary number fields. 

Finally, the second author and M.~Larsen \cite{IL13} introduced a combinatorial method to prove Conjecture~\ref{Larsen} for elliptic curves with full rational $2$-torsion. Recently, they further generalized this approach~\cite{IL26} to all elliptic curves over fields of characteristic zero unconditionally.

In this paper, we bridge Breuer's results and Larsen's Conjecture by proving Conjecture~\ref{Larsen} when $A=E$ is a non-isotrivial elliptic curve over a global function field $k$ of characteristic $p>3$, assuming that $E/k$ has analytic rank~$1$. More precisely, while Breuer~\cite{Bre04} established that the rank is infinite over the tower of ring class fields, we prove that the rank remains infinite even when we descend to the tower of the smaller fixed subfields of them under $G$.
We are now ready to state our first main result:

\begin{theorem}\label{thm:main}
    Let $E/k$ be a non-isotrivial elliptic curve over a global function field $k$ of characteristic $p>3$. Suppose the analytic rank of $E/k$ is $1$. Let $G$ be a topologically finitely generated subgroup of $G_k := \Gal(k^{\mathrm{sep}}/k)$. Then the rank of $E$ over $L^G$ is infinite, where $L$ is the union of all ring class extensions $K[\mathfrak{p}^n]$ for some finite separable extension $K/k$ and prime ideal $\mathfrak{p}$.
\end{theorem}

If $E/k$ has analytic rank $0$, the same conclusion is not known in general. Nevertheless, it follows under an additional assumption:

\begin{theorem}\label{analytic_rank_0}
    Let $E/k$ be a non-isotrivial elliptic curve over a global function field $k$ of characteristic $p>3$. Suppose the analytic rank of $E/k$ is $0$. Let $G$ be a topologically finitely generated subgroup of $G_k := \Gal(k^{\mathrm{sep}}/k)$. Assume, moreover, that there exists an imaginary quadratic extension $K/k$ such that $E/K$ has analytic rank 1 and $E/K$ satisfies the Heegner hypothesis. Then the rank of $E$ over $L^G$ is infinite, where $L$ is the union of all ring class extensions $K[\mathfrak{p}^n]$ for some prime ideal $\mathfrak{p}$.
\end{theorem}

Our work builds on the methods developed in earlier works~\cite{CI26} and~\cite{Im07}, which treat the number field setting. In contrast, we work over global function fields, where the structure of the ring class field tower exhibits fundamentally different behavior.

Let $K$ be an imaginary quadratic extension, and let $L$ denote the union of all ring class extensions $K[\mathfrak{p}^n]$ appearing in Theorem~\ref{thm:main}. By class field theory and the structure of local units (see \cite[Proposition~2.1]{Bre04}), the Galois group $\Gal(L/K)$ has the following structure:
\[
    \Gal(L/K)\cong
    \begin{cases}
        \{\text{finite group}\} \times \bbz_p, & \text{ if $k$ is a number field} \\
        \{\text{finite group}\} \times \bbz_p^\infty, & \text{ if $k$ is a function field}.
    \end{cases}
\]
Consequently, this structural difference introduces a notable shift in our argument. Since $\bbz_p$ is topologically generated by a single element, the first two authors used horizontal ring class field towers of the form $K[p_1\cdots p_n]$ in~\cite{CI26} where the $p_i$ are suitably chosen rational primes inert in $K$. In contrast, in the function field setting, the group $\bbz_p^\infty$ has infinitely many generators, which allows us to work directly with the vertical tower $K[\mathfrak{p}^n]$. As a result, our method avoids the auxiliary construction of infinite sequences of primes required in the number field case.

The second author previously used vertical ring class field towers in~\cite{Im07}, where Conjecture~\ref{Larsen} was established for elliptic curves when $G$ is procyclic, exploiting the dihedral structure of the tower without relying on the Gross–Zagier formula.
In the present work, we treat an arbitrary topologically finitely generated subgroup $G\subset G_k$, thereby obtaining a substantially more general result. Although our argument makes essential use of the Gross-Zagier-Zhang formula, the passage to general $G$ introduces additional technical complications. In particular, we obtain recursive relations not only for the Heegner points themselves but also for their traces (see Lemma~\ref{lem:vertical_recurrence}), which play a key role in our argument.

Another important difference lies in the existence of a suitable imaginary quadratic field $K$. In the number field case, classical results of \cite{BFH90} or \cite{MM91} guarantee the existence of an imaginary quadratic field $K/\bbq$ satisfying both the Heegner hypothesis and the non-vanishing of the relevant $L$-series. In the function field setting, however, only a weaker non-vanishing result is available due to Ulmer~\cite{Ulm05}. In particular, Ulmer's result provides such a quadratic extension only after passing to a finite separable extension $F$ of the base field (see Lemma~\ref{lem:cond_GZ}). To address this issue, in the analytic rank zero case, we impose the additional assumption that this auxiliary field $F$ coincides with the base field $k$.

\

The structure of this paper is as follows: in Section~\ref{sec_pre}, we introduce the foundational definitions of the $L$-functions associated to elliptic curves over function fields and Heegner systems. We also establish lemmas concerning the norm compatibilities of Heegner points and trace relations, needed for our main results. We then prove Theorem~\ref{thm:main} in Section~\ref{sec_analytic_rank_1}, followed by the proof of Theorem~\ref{analytic_rank_0} in Section~\ref{sec_analytic_rank_0}.


\section{Preliminaries}\label{sec_pre}

\subsection{Elliptic curves over function fields}

For a prime $p>3$, let $\mathbb{F}_q$ be a finite field of order $q$, where $q$ is a power of $p$. Let $\mathcal{C}$ be a geometrically connected smooth projective curve over  $\mathbb{F}_q$ and let $k=\mathbb{F}_q(\mathcal{C})$ be the corresponding global function field.

We recall several standard notions regarding elliptic curves over $k$. Let $E/k$ be an elliptic curve over $k$. The curve $E/k$ is called \emph{constant} if it is defined over $\mathbb{F}_q$. We say that $E$ is \emph{isotrivial} if it becomes isomorphic to a constant curve over a finite extension of $k$; equivalently, if its $j$-invariant $j(E)\in \mathbb{F}_q$. Finally, $E$ is said to be \emph{non-isotrivial} if $j(E)\notin \mathbb{F}_q$.

By the Mordell-Weil theorem \cite[Theorem~III.6.1]{Sil94}, the abelian group $E(k)$ of $k$-rational points over $k$ is finitely generated. Consequently, we can write 
\[E(k) \cong E(k)_{tors} \oplus \bbz^r,\]
where $r$ is a non-negative integer, called the algebraic rank of $E/k$.

Let $k_\infty$ denote the completion of $k$ at $\infty$ and write $\bbc_\infty$ for the completion of an algebraic closure of $k_\infty$. We denote by $\cO_k$ the ring of functions in $k$ regular outside $\infty$. This ring is a Dedekind domain with a finite class number
\[
h=|\operatorname{Pic}(\cO_k)|=\deg(\infty)h_k,
\]
where $h_k=|\operatorname{Pic}^0(k)|$ denotes the class number of $k$.

A Drinfeld module over a field $L$ is an $\mathbb{F}_q$-algebra homomorphism $\Phi:\cO_k\to L\{\tau\}$ satisfying natural conditions, where $\tau=\operatorname{Fr}(\tau_q)$ denotes the Frobenius endomorphism with coefficient in $L$. A cyclic $\mathfrak{n}$-isogeny is a morphism between Drinfeld modules whose kernel is a cyclic $\cO_k$-module of order $\mathfrak{n}$.
The Drinfeld modular curve $X_0(\mathfrak{n})$ parametrizes isomorphism classes of pairs $(\Phi,\Phi')$ of rank-$2$ Drinfeld $\cO_k$-modules linked by a cyclic $\mathfrak{n}$-isogeny.

The conductor of $E$ is of the form $\mathfrak{n}\infty$, where $\mathfrak{n}\subseteq \cO_k$ is an ideal. By the work of Drinfeld, as asserted in \cite[Section 8]{GR96}, there exists a modular parametrization
\[
\pi: X_0(\mathfrak{n})\to E
\]
defined over $k$.

For a finite place $v$ of $k$, denote by $k_v$ its residue field and by  $q_v$ its cardinality. If $E/k$ has  good reduction at $v$, then we define 
\[a_v := q_v+1 - \lvert E_v(k_v) \rvert\]
and
\[L_v(E/k,T) := 1-a_vT+q_vT^2.\] 
If $E$ has bad reduction at $v$, we define $L_v(T)$ as follows:
\[
L_v(E/k,T) :=
\begin{cases}
    1-T, & \text{ if $E$ has split multiplicative reduction at $v$} \\
    1+T, & \text{ if $E$ has non-split multiplicative reduction at $v$} \\
    1, & \text{ if $E$ has additive reduction at $v$.}
\end{cases}
\]
The $L$-function of $E/k$ is defined by the Euler product
\begin{align*}
    L(E/k,s) &:= \prod_{v} L_v(E/k,q_v^{-s})^{-1} \\
    &= \prod_{v\nmid \mathfrak{n}} \left(1-a_vq_v^{-s}+q_v^{1-2s}\right)^{-1}\times \prod_{v\mid \mathfrak{n}} L_v(E/k,q_v^{-s})^{-1}.
\end{align*}
By the Hasse bound $|a_v|\leq 2\sqrt{q_v}$, this product converges absolutely for $\operatorname{Re}(s)>3/2$. It admits a meromorphic continuation to the entire complex plane and satisfies a functional equation relating $s$ and $2-s$.

The analytic rank of $E/k$ is defined as the order of vanishing of $L(E/k,s)$ at $s=1$. The Birch-Swinnerton-Dyer conjecture predicts that the algebraic rank of $E/k$ and the analytic rank of $E/k$ are equal:
\begin{equation}\label{BSD_equality}
    \mathrm{rank}_\bbz E(k) = \mathrm{ord}_{s=1} L(E/k,s).
\end{equation}
In general, the inequality 
\[\mathrm{rank}_\bbz E(k) \leq \mathrm{ord}_{s=1} L(E/k,s)\]
is known to be true, and the equality \eqref{BSD_equality} is known to hold when the analytic rank of $E/k$ is~$\leq 1$.

Let $\chi$ be a finite order Hecke character of $k$. The twisted $L$-function $L(E/k,\chi,s)$ is defined by modifying the local Euler factors at good places $v\nmid \mathfrak{n}$ as
\[
L_v(E/k,\chi,q_v^{-s})^{-1} = 1-a_v\chi(v)q_v^{-s}+\chi(v)^2q_v^{1-2s},
\]
and similarly at places of bad reduction.This Euler product converges in the same region as $L(E/k,s)$, admits a meromorphic continuation to the entire complex plane, and satisfies a function equation.

Let $K/k$ be a quadratic extension, and let $\eta$ be the associated quadratic Hecke character. We have the factorization
\[
L(E/K,s) = L(E/k,s)L(E/k,\eta,s).
\]
Consequently, the analytic rank over $K$ is the sum of the orders of vanishing at $s=1$ of the two factors. This will be used to compare the analytic rank over $k$ and over quadratic extensions.

\subsection{Heegner systems}
Let $E/k$ be a non-isotrivial elliptic curve over a function field $k$ with conductor $\mathfrak{n}\infty$ and let $K$ be an imaginary quadratic extension of $k$ (i.e., the place $\infty$ does not split in $K/k$) such that the pair $(E,K)$ satisfies the Heegner hypothesis, as recalled below:

\begin{namedtheorem}[Heegner hypothesis] \cite[(3.2.2)]{Bro25}
    All places dividing $\mathfrak{n}$ split completely in $K/k$.
\end{namedtheorem}

We let $\cO_K$ denote the integral closure of $\cO_k$ in $K$.
Let $\mathfrak{p}\subset \cO_k$ be a non-zero place not dividing $\mathfrak{n}$, and define the order of conductor $\mathfrak{p}^n$ in $\cO_K$ by
\[
\cO_n:=\cO_k + \mathfrak{p}^n\cO_K, \qquad n\geq 1.
\]
Let $K[\mathfrak{p}^n]$ denote the ring class field of $K$ of conductor $\mathfrak{p}^n$, and set
\[
K[\mathfrak p^\infty] := \bigcup_{n \ge 1} K[\mathfrak p^n].
\]

By the Heegner hypothesis, there exists an ideal $\mathfrak{R}\subseteq \cO_K$ such that $\cO_K/\mathfrak{R}\simeq \cO_k/\mathfrak{n}$. Setting $\mathfrak{R}_n:=\mathfrak{R}\cap \cO_n$, we obtain
\[
\cO_n/\mathfrak{R}_n\simeq \cO_k/\mathfrak{n}
\]
Thus, we may regard $\cO_n$ and $\mathfrak{R}_n^{-1}$ as rank-$2$ $\cO_k$-lattices in $\bbc_\infty$. This yields a pair of Drinfeld modules $(\Phi^{\cO_n},\Phi^{\mathfrak{R}_n^{-1}})$ linked by a cyclic $\mathfrak{n}$-isogeny. 
Hence, they define a \emph{Heegner point}~$x_n$ on $X_0(\mathfrak{n})$, which is defined over the ring class field $K[\mathfrak{p}^n]$.
Under the modular parametrization $\pi:X_0(\mathfrak{n})\to E$, we obtain
points $y_n:=y_{\mathfrak{p}^n}=\pi(x_{\mathfrak{p}^n})\in E(K[\mathfrak{p}^n])$.

To describe the Galois structure of the tower, we first recall a result of Breuer on the structure of the Galois group:
\begin{lemma} \cite[Proposition 2.1]{Bre04} \label{lem:Galois_group_decomposition}
    $G := \Gal(K[\mathfrak{p}^\infty]/K) \simeq \bbz_p^\infty \times G_0$, where $G_0:=G_{\operatorname{tors}}$ is a~finite group.
\end{lemma}

Since we are mainly interested in the $p$-power part of the tower, we may decompose
\[
\mathrm{Gal}(K[\mathfrak p^n]/K)
\simeq
(\text{$p$-primary part}) \times (\text{prime-to-$p$ part}),
\]
and restrict our attention to the Sylow $p$-subgroup. In particular, the Heegner system naturally lives over the pro-$p$ extension $K[\mathfrak p^\infty]$.
By class field theory, we have
\begin{equation} \label{eq:Galois_class field tower}
    \Gal(K[\mathfrak{p}^n]/K)\simeq \operatorname{Pic}(\cO_n).
\end{equation}

The collection $\{y_n\}_{n\geq 1}$ satisfying certain conditions forms a Heegner system attached to the pair $(E,K)$. For details, see \cite[Section 6.5]{Bro25}. Under the Heegner hypothesis, the existence of a non-trivial Heegner system $\{y_n\}$ follows from \cite[Section 4.2]{BI08} and the following lemma.

\begin{lemma}[norm-compatibilities] \cite[Table 3.4.8]{Bro25} \label{norm-compatibility}
    Let $\mathfrak{p}\subset \cO_k$ be a non-zero place coprime to~$\mathfrak{n}$.
    Then, the Heegner system $\{y_n\}_{n\geq 1}$ attached to $(E,K)$ satisfies the following trace relations:
    \begin{equation}
        \tr_{K[\mathfrak{p}]/K[1]}(y_1) = a_{\mathfrak{p}}y_0
    \end{equation}
    and
    \begin{equation} \label{eq:second norm-compatibility}
        \tr_{K[\mathfrak{p}^{n+1}]/K[\mathfrak{p}^n]}(y_{n+1}) = a_{\mathfrak{p}}y_n - y_{n-1},
    \end{equation}
    where $a_{\mathfrak{p}}=\tr(\operatorname{Fr}_{\mathfrak{p}}\mid T_\ell(E))$ is the trace of Frobenius acting on the $\ell$-adic Tate module $T_\ell(E)$ for a prime $\ell\neq \operatorname{char}(k)$. 
\end{lemma}

\subsection{Function field analogue of Gross-Zagier-Zhang formula}
The N{\'e}ron-Tate height pairing provides a $\bbq$-linear non-degenerate pairing
\[
\langle \cdot, ~ \cdot \rangle_{\mathrm{NT}}: E(F^{\mathrm{sep}})_{\bbq}\times E(F^{\mathrm{sep}})_{\bbq} \to \bbc.
\]

Let
\[
M=\mathrm{End}(E)_{\bbq}:=\mathrm{End}(E)\otimes_{\bbz} \bbq.
\]
The field $M$ acts naturally on $E(F^{\mathrm{sep}})_{\bbq}$. The compatibility of the N{\'e}ron–Tate pairing with dual isogenies implies that
\[
\langle aP,Q\rangle_{\mathrm{NT}}
=
\langle P,\hat{a} Q\rangle_{\mathrm{NT}}
\qquad (a\in M),
\]
so the pairing is $M$-balanced. Consequently, it descends to
\[
\langle \cdot, ~ \cdot \rangle_{\mathrm{NT}}: E(F^{\mathrm{sep}})_{\bbq}\otimes_M E(F^{\mathrm{sep}})_{\bbq} \to \bbc.
\]
Using the trace map $\mathrm{Tr}_{M\otimes \bbc/\bbc}$, we extend this to an $M$-bilinear pairing
\[
\langle \cdot, ~ \cdot \rangle_{\mathrm{NT}}^{M}: E(F^{\mathrm{sep}})_{\bbq}\otimes_M E(F^{\mathrm{sep}})_{\bbq}\to M\otimes_{\bbq} \bbc
\]

Let $y_0$ be the Heegner point of conductor $1$ defined over the ring class field $K[1]$ of $K$. We denote its trace to $K$ by
\[
P_K := \Tr_{K[1]/K}(y_0)\in E(K)
\]
By the global Langlands correspondence for $\mathrm{GL}_2$ over function fields proved by Drinfeld~\cite{Dri77}, $E$ corresponds to a cuspidal automorphic representation $\pi_E$ whose automorphic $L$-function agrees with the Hasse-Weil $L$-function of $E$ up to normalization.

Let $\chi:\Gal(K/k)\to \bbc^\times$ be a character.
Fix an automorphic vector $\phi\in \pi_E$. Following the automorphic construction of Heegner points, we define the $\chi$-twisted point by
\[
P_{\chi}=P_{\chi}(\phi):=\sum_{\sigma\in \Gal(K[\mathfrak{p}]/K)}\chi(\sigma)y_0^\sigma\in E(K)\otimes \bbc.
\]
This construction corresponds to the toric period appearing in the function field analogue of the automorphic Gross-Zagier-Zhang formula proved by Qiu~\cite{Qiu22}, reduced to the elliptic curve cases:
\begin{lemma} \cite[Theorem 1.2.1]{Qiu22} \label{lem:GZ_for_FF}
    For any $\phi\in \pi_E$, we have
    \[
    \langle P_{\chi}, P_{\chi}\rangle_{\mathrm{NT}} = \frac{L(2,1_k)L'(1/2,\pi_E,\chi)}{4L(1,\eta)^2L(1,\pi_E,\mathrm{ad})} \alpha_{\pi_E}(\phi,\phi).
    \]
    In particular,
    \[
    P_\chi \text{ is non-torsion} \text{ if and only if } 
    L'(E/K,1,\chi) \text{ is non-vanishing}.
    \]
\end{lemma}
\begin{proof}
    The function field analogue of the Gross-Zagier-Zhang formula is established by Qiu~\cite[Theorem 1.2.1]{Qiu22} for abelian varieties. For an abelian variety $A$ and vectors $\phi\in\pi_A$ and $\varphi\in\pi_{A^\vee}$, Qiu proves that
    \[
    \langle P_{\Omega}(\phi), P_{\Omega^{-1}}(\varphi)\rangle_{\mathrm{NT}}^{L'} = \frac{L(2,1_k)L'(1/2,\pi_A,\Omega)}{4L(1,\eta)^2L(1,\pi_A,\mathrm{ad})} \alpha_{\pi_A}(\phi,\varphi) \in K'\otimes_{\bbq} \bbc,
    \]
    We specialize this formula to the elliptic curve $E$.
    In Qiu's construction, we have 
    \[
    P_{\Omega}(\phi) = \int_{\Gal(E^{\mathrm{ab}}/E)} \phi(P_0)^\tau \otimes \Omega(\tau) d\tau,
    \]
    where $P_0$ is a CM point.

    In our setting, the Heegner point $y_0$ is defined over the ring class field $K[1]$, so the above Galois integral factors through the finite group $\Gal(K[1]/K)$. Twisting by $\Omega$ corresponds to twisting by the character $\chi$ on $\Gal(K[1]/K)$, and therefore the automorphic point $P_{\Omega}(\phi)$ coincides with the twisted point $P_{\chi}$.

    Taking $\varphi=\phi$ and noting that $\Omega^{-1}=\Omega$ for quadratic characters, Qiu's formula specializes to the desired form.

    It remains to check that the other factors are non-zero.
    Since $\eta$ is a non-trivial quadratic character of $K/k$, the $L$-function $L(s,\eta)$ is entire in our setting, so $L(1,\eta)\neq 0$.
    The adjoint $L$-function of $\pi_E$ is defined by
    \[
    L(s,\pi_E,\mathrm{ad}) := \frac{L(s,\pi_E\times \pi_E^{\vee})}{L(s,1_k)} = L(s,\mathrm{Sym}^2\pi_E)
    \]
    By the symmetric square lifting of Gelbart and Jacquet~\cite[Theorem 9.3]{GJ78}, this $L$-function is holomorphic and non-vanishing at $s=1$.

    Since almost all places $v$ are unramified, the local factor $\alpha_{\pi_v}(\phi_v,\phi_v)$ equals $1$. At the finitely many ramified places, one can choose $\phi_v\neq 0$ so that the local toric integral over $F_v^\times/k_v^\times$ (which is essentially a finite sum) is non-zero.
    Hence,
    \[
    \alpha_{\pi_E}(\phi,\phi)
    =\prod_v \alpha_{\pi_v}(\phi_v,\phi_v)\neq 0.
    \]

    For the Dedekind zeta function
    \[
    L(s,1_k)=\prod_v (1-q_v^{-s})^{-1},
    \]
    the Euler product converges absolutely for $\Re(s)>1$. 
    In particular, at $s=2$ each Euler factor $(1-q_v^{-2})^{-1}$ is positive, so the product converges to a positive real number.
    Hence, $L(2,1_k)\neq 0$.
    
    Finally, we relate the automorphic and geometric $L$-functions. The automorphic $L$-function $L(s,\pi_E,\chi)$ is normalized so that its functional equation exchanges $s$ and $1-s$, and therefore its central point is $s=1/2$. On the other hand, the Hasse-Weil $L$-function $L(E/K,s,\chi)$ satisfies a functional equation relating $s$ and $2-s$, whose central point is $s=1$. These normalizations differ by the shift $s\mapsto s+1/2$, under which $L'(1/2,\pi_E,\chi)$ corresponds to $L'(E/K,1,\chi)$.
    This proves the stated formula and the equivalence between the
    non-torsion of $P_\chi$ and the non-vanishing of $L'(1/2,\pi_E,\chi)$.    
\end{proof}

\begin{remark}
In the special case $k=\mathbb{F}_q[T]$, the formula was previously established by R{\"u}ck and Tipp~\cite[Theorem 4.2.1]{RT00} in the modular framework.
\end{remark}

\subsection{Constructing imaginary quadratic fields}
Let $E/k$ be a non-isotrivial elliptic curve over $k$ of conductor $\mathfrak{n}\infty$. 

Ulmer's geometric non-vanishing result~\cite[Theorem 1.2]{Ulm05} guarantees the existence of an imaginary quadratic extension of a suitable finite separable extension of $k$ satisfying the Heegner hypothesis.

This provides precisely the setting needed to state the following lemma for applying Lemma~\ref{lem:GZ_for_FF}:

\begin{lemma} \label{lem:cond_GZ}
    Suppose that $k=\mathbb{F}_q(\mathcal{C})$ has characteristic $p>3$. Let $E/k$ be a non-isotrivial elliptic curve over $k$ of conductor $\mathfrak{n}\infty$, and suppose that $E/k$ has analytic rank at most $1$. Then, there exists a finite separable extension $F$ of $k$ and an imaginary quadratic extension~$K$ of $F$ satisfying the following conditions:
    \begin{enumerate}[\normalfont(i)]
        \item $(E,K)$ satisfies the Heegner hypothesis.
        \item $L'(E/K,1)$ is non-vanishing.
    \end{enumerate}
\end{lemma}

\begin{proof}
    By~\cite[Theorem 1.2(a)--(d)]{Ulm05}, the quadratic extension $K/F$ can be chosen so that the Heegner hypothesis (i) is satisfied. Furthermore, \cite[Theorem 1.2(e)]{Ulm05} guarantees that $L'(E/K,1)$ is non-vanishing, establishing (ii).
\end{proof}

\subsection{Elementary lemmas for ranks of elliptic curves}

Let $E/k$ be an elliptic curve over a field $k$, and let $K/k$ be a finite separable extension. Given a point $P \in E(K)$, the trace $\Tr_{K/k}(P) \in E(k)$ is defined by the formula 
\[\Tr_{K/k}(P) := \sum_{\sigma} \sigma(P)\]
where the sum runs over all embeddings of $K$ into $k^{\mathrm{sep}}$ fixing $k$.

\begin{lemma}\label{same_rank}
    Let $K/k$ be a finite separable extension of a global function field $k$ and let $E/k$ be an elliptic curve. Assume that
    \[\operatorname{rank} E(k) = \operatorname{rank} E(K) = 1.\]
    Then, for any non-torsion point $P\in E(K)$, the trace $\Tr_{K/k}(P)$ is also non-torsion in $E(k)$.
\end{lemma}
\begin{proof}
    Since $E(k) \subseteq E(K)$ and both have rank $1$, after tensoring with $\bbq$, we obtain the equality,
    \begin{equation}\label{rank_1_equality}
        E(k)\otimes_{\bbz} \bbq = E(K)\otimes_{\bbz} \bbq.
    \end{equation}
    Let $P\in E(K)$ be a non-torsion point. Then, \eqref{rank_1_equality} above yields a rational point $Q \in E(k)$ and an integer $m \geq 1$ such that
    \[mP \equiv Q \quad \text{modulo }E(K)_{\mathrm{tors}}.\]
    By appropriately replacing $m$ and $Q$ to eliminate the torsion part, we may assume that
    \[mP=Q.\]
    Note that this forces $Q$ to be non-torsion.
    
    Applying the trace map $\Tr_{K/k}: E(K)\to E(k)$ which is $\bbz$-linear, we have 
    \[m\Tr_{K/k}(P) = \Tr_{K/k}(mP)=\Tr_{K/k}(Q) = [K:k]Q.\]
    Therefore, $\Tr_{K/k}(P)$ cannot be torsion in $E(k)$, completing the proof.
\end{proof}

We will also need the following lemma regarding the rank of an elliptic curve over an infinite field extension.

\begin{lemma} \cite[Lemma 3.3]{BI08}
    Let $E/k$ be an elliptic curve and $L/k$ be a Galois extension of $k$, and let $\{P_m\}$ be an infinite sequence of points in $E(L)$. Denote by $\mathcal{S}$ the subgroup of $E(L)$ generated by the points $P_m$. Suppose that:
    \begin{enumerate}[\normalfont(i)]
        \item $E(L)_{\operatorname{tors}}$ is finite, and
        \item $\mathcal{S}$ is not finitely generated.
    \end{enumerate}
    Then, we have
    \[
    \dim \mathcal{S}\otimes \bbq = \dim E(L)\otimes \bbq = \infty.
    \]
\end{lemma}

\section{Proof of Theorem~\ref{thm:main}}\label{sec_analytic_rank_1}

In this section, we prove Theorem~\ref{thm:main}. Throughout this section, we assume that:
\begin{itemize}
    \item $k$ is a global function field of characteristic $p>3$,
    \item $E/k$ is a non-isotrivial elliptic curve,
    \item $E/k$ has analytic rank $1$,
    \item $G \subset G_k$ is a topologically finitely generated closed subgroup.
\end{itemize}
By Lemma~\ref{lem:cond_GZ}, there exists a finite separable extension $F/k$ and a quadratic extension $K/F$ such that both $E/F$ and $E/K$ have analytic rank~$1$, and such that  $E/K$ satisfies the Heegner hypothesis. By Lemma~\ref{lem:GZ_for_FF}, the Heegner point $P_K \in E(K)$ is non-torsion. By Lemma~\ref{same_rank}, the trace 
\[P_k := \Tr_{K/k}(P_K) \in E(k)\] 
is also non-torsion.

Recall that we have higher Heegner points $y_n \in E(K[\mathfrak{p}^n])$. We compute their full trace $\Tr_{K[\mathfrak{p}^n]/K[1]}(y_n)$ by iterating the norm-compatibility~\eqref{eq:second norm-compatibility}.

\begin{lemma}\label{lem:vertical_recurrence}
    For all $n \geq 0$, 
    \[\Tr_{K[\mathfrak{p}^n]/K[1]}(y_n) = A_ny_0,\]
    where $A_n$ is defined recursively by
    \begin{equation} \label{eq:coefficient_recurrence}
      A_0 = 1, \quad A_1=a_\mathfrak{p}, \quad   A_{n+1} = a_\mathfrak{p}A_n - q_\mathfrak{p}A_{n-1},  \text{ for } n \geq 1.
    \end{equation}
\end{lemma}
\begin{proof}
    The base cases $n=0$ and $n=1$ are clear. Let $n \geq 1$. By the norm-compatibility \eqref{eq:second norm-compatibility},
    \[\Tr_{K[\mathfrak{p}^{n+1}]/K[1]}(y_{n+1}) = \Tr_{K[\mathfrak{p}^{n}]/K[1]}(a_{\mathfrak{p}}y_n - y_{n-1}).\]
    By the induction hypotheses, we have
    \begin{align*}
        \Tr_{K[\mathfrak{p}^{n+1}]/K[1]}(y_{n+1}) &= a_\mathfrak{p}A_n y_0 - [K[\mathfrak{p}^n]:K[\mathfrak{p}^{n-1}]]A_{n-1}y_0 \\
        &= (a_\mathfrak{p}A_n - q_\mathfrak{p}A_{n-1})y_0.
    \end{align*}
    This yields the desired recurrence, completing the induction.
\end{proof}

Since $\Tr_{K[1]/k}(y_0) = P_k$, we conclude that 
\begin{equation*}
    \Tr_{K[\mathfrak{p}^n]/k}(y_n) = A_nP_k,\quad n \geq 0.
\end{equation*}
Recall that $P_k$ is non-torsion.

Note that the recurrence~\eqref{eq:coefficient_recurrence} shows that $A_n$ grows exponentially like~$\lambda^n$, where $\lambda$ is a Frobenius eigenvalue at $\mathfrak{p}$. Explicitly, $\lambda$ is a root of 
\[x^2-a_\mathfrak{p}x + q_\mathfrak{p} = 0.\]
Thus, 
\[\lvert \lambda \rvert \ll a_\mathfrak{p} \ll \sqrt{q_\mathfrak{p}}\]
by the Hasse bound~\cite[Theorem V.1.1]{Sil09}. Then, 
\begin{equation}\label{A_n_size}
    \lvert A_n \rvert \ll \lvert \lambda \rvert^n \ll q_\mathfrak{p}^{n/2}.
\end{equation}
We next want to investigate the size of the $p$-parts of $[K[\mathfrak{p}^n]:k]$ and $[K[\mathfrak{p}^n]^G:k]$. Let $H = G \cap G_{K[1]}$. Since $H$ is an open subgroup of $G$, $H$ is finitely generated, say by $r$ elements.

\begin{lemma}\label{p-part_size}
    We have 
    \[v_p([K[\mathfrak{p}^n]:k]) = (n-1)f + v_p([K[1]:k]),\]
    and 
    \[v_p([K[\mathfrak{p}^n]^G:k]) \geq (n-1)f - r\lceil \log_p n \rceil,\]
    where $q_\mathfrak{p} = p^f$.
\end{lemma}
\begin{proof}
    Note that $H$ is an open subgroup of $G$ of index dividing $[K[1]:k]$. Therefore, it suffices to prove 
    \begin{equation}\label{p-part1}
        v_p([K[\mathfrak{p}^n]:K[1]]) = f(n-1)
    \end{equation}
    and 
    \begin{equation}\label{p-part2}
        v_p([K[\mathfrak{p}^n]^H:K[1]]) \geq f(n-1) - r\lceil \log_p n \rceil.
    \end{equation}
By \cite[(2.3.8)]{Bro04}, 
    \begin{equation}\label{ring_class_field_galois_group}
        \Gal(K[\mathfrak{p}^n]/K[1]) \cong \frac{(B/\mathfrak{p}^nB)^\times/(A/\mathfrak{p}^nA)^\times}{B^\times/A^\times},
    \end{equation}
    where $A$ and $B$ are rings of integers of $k$ and $K$, respectively. By \cite[(2.3.12)]{Bro04}, $\Gal(K[\mathfrak{p}]/K[1])$ is cyclic of order prime to $p$. On the other hand, for $k \geq 1$, $\Gal(K[\mathfrak{p}^{k+1}]/K[\mathfrak{p}^{k}])$ is a direct sum of cyclic groups of order $p$, and its degree is exactly $q_\mathfrak{p}$. Hence, the Sylow $p$-subgroup of $\Gal(K[\mathfrak{p}^n]/K[1])$ has order $q_\mathfrak{p}^{n-1}$, which proves \eqref{p-part1}.
    
    We now want to investigate the $p$-exponent of $\Gal(K[\mathfrak{p}^n]/K[1])$, namely the exponent of its Sylow $p$-subgroup. By~\eqref{ring_class_field_galois_group}, this is clearly bounded by the $p$-exponent of $(B/\mathfrak{p}^nB)^\times$. 
    
    Recall the exact sequence 
    \[1 \longrightarrow (1+\mathfrak{p})/(1+\mathfrak{p}^n) \longrightarrow (A/\mathfrak{p}^n)^\times \longrightarrow (A/\mathfrak{p})^\times \longrightarrow 1.\]
    Since $A/\mathfrak{p}$ is a field of cardinality $q_\mathfrak{p}$, $(A/\mathfrak{p})^\times$ has cardinality $q_\mathfrak{p}-1$, which is relatively prime to $p$. Therefore, the $p$-part of $(A/\mathfrak{p}^n)^\times$ equals that of $(1+\mathfrak{p})/(1+\mathfrak{p}^n)$. Since the latter group has $p$-exponent $\leq p^{\lceil \log_p n \rceil}$, $(A/\mathfrak{p}^n)^\times$ also has $p$-exponent $\leq p^{\lceil \log_p n \rceil}$.

    Since $\mathfrak{p}$ is either inert or split in $B$, the above argument implies that $(B/\mathfrak{p}^nB)^\times$ has $p$-exponent $\leq p^{\lceil \log_p n \rceil}$. Therefore, $\Gal(K[\mathfrak{p}^n]/K[1])$ also has $p$-exponent at most $\leq p^{\lceil \log_p n \rceil}$. 

    Since $G$ is generated by $r$ elements, and the Sylow $p$-subgroup of $\Gal(K[\mathfrak{p}^n]/K[\mathfrak{p}^n]^G)$ has exponent at most $p^{\lceil \log_p n \rceil}$, it follows that the order of the Sylow $p$-subgroup of $\Gal(K[\mathfrak{p}^n]/K[\mathfrak{p}^n]^G)$ is at most $p^{r\lceil \log_p n \rceil}$.
    On the other hand, we have already observed that the Sylow $p$-subgroup of $\Gal(K[\mathfrak{p}^n]/K[1])$ has order $q_\mathfrak{p}^{n-1}$. Therefore, the Sylow $p$-subgroup of $\Gal(K[\mathfrak{p}^n]^G/K[1])$ has order at least $p^{f(n-1)-r\lceil \log_p n \rceil}$, which proves \eqref{p-part2}.
\end{proof}

We are now ready to complete the proof of Theorem~\ref{thm:main}. Define 
\[z_n := \Tr_{K[\mathfrak{p}^n]/K[\mathfrak{p}^n]^G}(y_n) \in E(K[\mathfrak{p}^n]^G),\quad n \geq 0.\]

Suppose, for the sake of contradiction, that the set
\[
\left\{z_j\right\}_{j\geq r}
\]
generates a finitely generated subgroup of $E(k^{\mathrm{sep}})$. Since each generator is defined over a finite extension of $k$, there exists an integer $m$ such that all $z_j$ are contained in $K[\mathfrak{p}^m]$.

For $n\geq m$, we have
\[
\Tr_{K[\mathfrak{p}^n]^G/k}(z_n) = \Tr_{K[\mathfrak{p}^n]/k}(y_n) = A_nP_k.
\]
On the other hand, by the transitivity of trace,
\begin{align*}
    \Tr_{K[\mathfrak{p}^n]^G/k}(z_n) &= \Tr_{K[\mathfrak{p}^m]^G/k}\Tr_{K[\mathfrak{p}^n]^G/K[\mathfrak{p}^m]^G}(z_n) \\
    &= [K[\mathfrak{p}^n]^G:K[\mathfrak{p}^m]^G]\Tr_{K[\mathfrak{p}^m]^G/k}(z_n).
\end{align*}

Let $\{Q\}$ be a basis of $E(k)$ and we write 
\[P_k = cQ,\quad \Tr_{K[\mathfrak{p}^m]^G/k}(z_n) = c_nQ.\]
Comparing the two expressions for $\Tr_{K[\mathfrak{p}^n]^G/k}(z_n)$ yields
\begin{equation}\label{contradiction_equality}
    [K[\mathfrak{p}^n]^G:K[\mathfrak{p}^m]^G]c_n = A_nc.
\end{equation}
By Lemma~\ref{p-part_size}, we have 
\[v_p([K[\mathfrak{p}^n]^G:K[\mathfrak{p}^m]^G]) \geq f(n-m)-r\lceil \log_p n \rceil - v_p([K[1]:k]) \gg (1-\varepsilon)fn,\]
for fixed $0<\varepsilon<1/2$. Thus, for sufficiently large $n$,
\[[K[\mathfrak{p}^n]^G:K[\mathfrak{p}^m]^G] \gg p^{(1-\varepsilon)fn}.\]
On the other hand, from the growth of $|A_n|$ \eqref{A_n_size}, we have
$\lvert A_n \rvert \ll q_\mathfrak{p}^{n/2} = p^{(1/2)fn}$ for large $n$.
Therefore, \eqref{contradiction_equality} yields 
\[\lvert c \rvert \gg p^{(1/2-\varepsilon)fn}.\]
Since $c$ is a constant independent of $n$, letting $n \to \infty$ yields a contradiction. This completes the proof of  Theorem~\ref{thm:main}.

\section{Proof of Theorem~\ref{analytic_rank_0}}\label{sec_analytic_rank_0}

In this section, we prove Theorem~\ref{analytic_rank_0}. Throughout this section we assume that:
\begin{itemize}
    \item $k$ is a global function field of characteristic $>3$,
    \item $E/k$ is a non-isotrivial elliptic curve,
    \item $E/k$ has analytic rank $0$,
    \item $G \subset G_k$ is a topologically finitely generated closed subgroup.
\end{itemize}
and we assume, moreover, that there exists an imaginary quadratic extension $K/k$ such that:
\begin{enumerate}
    \item $E/K$ has analytic rank $1$,
    \item $E/K$ satisfies the Heegner hypothesis.
\end{enumerate}

By the Gross--Zagier formula, the Heegner point
\[P_K \in E(K)\]
is non-torsion.
Since $E/k$ has analytic rank $0$, the Mordell--Weil group $E(k)$ is finite. Hence, the trace
\[P_k := \mathrm{Tr}_{K/k}(P_K)\]
is torsion.

This condition seemingly appears to preclude the direct application of the argument from Section~\ref{sec_analytic_rank_1}, as it necessitates a non-torsion trace. To overcome this difficulty, we employ the strategy introduced in \cite[Section~4.2]{CI26} as follows.

Let $H = G \cap G_K$, and let
\[y_n \in E(K[\mathfrak p^n])\]
be the Heegner points. Define
\[z_n := \Tr_{K[\mathfrak p^n]/K[\mathfrak p^n]^H}(y_n) \in E\bigl(K[\mathfrak p^n]^H\bigr).\]
Using the same argument as in Section~\ref{sec_analytic_rank_1}, we can show that the set $\{z_n\}_{n \ge 1}$ generates a subgroup of infinite rank in $E\bigl((k^{\mathrm{sep}})^H\bigr)$.

If $H=G$, then there is nothing to prove; therefore, we may assume $H \neq G$ and choose an element $\sigma \in G  ~\setminus~ H$. Then $[G:H]=2$ and under the projection $\Gal(k^{\mathrm{sep}}/k) \rightarrow \Gal(K[\mathfrak{p}^n]/k)$, the images of $G$ and $H$ satisfy  $G=H\rtimes\langle\sigma\rangle$.

Note that $K[\mathfrak p^n]/K$ is abelian, and $K[\mathfrak p^n]/k$ is Galois with a dihedral Galois group:
\[\Gal(K[\mathfrak p^n]/k) \cong \Gal(K[\mathfrak p^n]/K)\rtimes \langle \sigma \rangle,\]
where
\[\sigma \tau \sigma^{-1} = \tau^{-1}\qquad(\tau \in \Gal(K[\mathfrak p^n]/K)).\]
Now, applying the same argument as in \cite[Section~4.2]{CI26},  we conclude that $E\bigl((k^{\mathrm{sep}})^G\bigr)$ also has infinite rank, thereby proving the theorem.


\end{document}